\documentclass[10pt]{article}

\usepackage[utf8]{inputenc}
\usepackage[T1]{fontenc}
\usepackage{color}
\usepackage{amsmath}
\usepackage{amssymb}
\numberwithin{equation}{section}
\usepackage{microtype}
\usepackage{graphicx}
\graphicspath{{images/}}
\usepackage{tikz}
\usepackage{pgfplots}
\usepackage[titletoc]{appendix}
\pgfplotsset{compat=newest}

\usepackage[backend=biber,maxnames=10,backref=true,hyperref=true,safeinputenc]{biblatex}
\DefineBibliographyStrings{english}{%
	backrefpage = {cited on page},
	backrefpages = {cited on pages},
}

\usepackage[pdftex,colorlinks=true,linkcolor=blue,citecolor=green,urlcolor=blue,bookmarks=true,bookmarksnumbered=true]{hyperref}

\usepackage[a4paper,centering,bindingoffset=0cm,marginpar=2cm,margin=2.5cm]{geometry}
\usepackage[pagestyles]{titlesec}
\titleformat{\section}[block]{\large\sc\filcenter}{\thesection.}{0.5ex}{}[]
\titleformat{\subsection}[runin]{\bf}{\thesubsection.}{0.5ex}{}[.]
\newpagestyle{headers}%
{%
  \headrule
  \sethead%
     [\footnotesize\thepage]%
     [\footnotesize\sc Authors]%
     []%
     {}%
     {\footnotesize\sc }%
     {\footnotesize\thepage}
  \setfoot{}{}{}
}
\usepackage[font=footnotesize,format=plain,labelfont=sc,textfont=sl,width=0.75\textwidth,labelsep=period]{caption}

\title{Preconditioning Inverse Problems for Hyperbolic Equations with Applications to Photoacoustic Tomography}
\author{Alexander Beigl$^1$\\{\footnotesize\href{mailto:alexander.beigl@univie.ac.at}{alexander.beigl@univie.ac.at}}
\and Otmar Scherzer$^{1,3}$\\{\footnotesize\href{mailto:otmar.scherzer@univie.ac.at}{otmar.scherzer@univie.ac.at}}
\and Jarle Sogn$^2$\\{\footnotesize\href{mailto:jarle@numa.uni-linz.ac.at}{jarle@numa.uni-linz.ac.at}}
\and Walter Zulehner$^2$\\{\footnotesize\href{mailto:zulehner@numa.uni-linz.ac.at}{zulehner@numa.uni-linz.ac.at}}}
\date{May 29, 2019}

\usepackage[hyperref,amsmath,thmmarks]{ntheorem}
\usepackage{aliascnt}
\usepackage{framed}

\newtheorem{lemma}{Lemma}[section]

\newaliascnt{proposition}{lemma}

\aliascntresetthe{proposition}

\newaliascnt{corollary}{lemma}
\newtheorem{corollary}[corollary]{Corollary}
\aliascntresetthe{corollary}

\newaliascnt{theorem}{lemma}
\newtheorem{theorem}[theorem]{Theorem}
\aliascntresetthe{theorem}

\theorembodyfont{\normalfont}
\newaliascnt{definition}{lemma}
\newtheorem{definition}[definition]{Definition}
\aliascntresetthe{definition}

\newaliascnt{assumption}{lemma}

\aliascntresetthe{assumption}

\newaliascnt{hypothesis}{lemma}

\aliascntresetthe{hypothesis}

\newaliascnt{problem}{lemma}
\newtheorem{problem}[problem]{Problem}
\aliascntresetthe{problem}

\newaliascnt{notation}{lemma}
\newtheorem{notation}[notation]{Notation}
\aliascntresetthe{notation}

\newaliascnt{example}{lemma}

\aliascntresetthe{example}

\newaliascnt{remark}{lemma}
\newtheorem{remark}[remark]{Remark}
\aliascntresetthe{remark}

\theoremstyle{nonumberplain}
\theoremseparator{:}
\theoremheaderfont{\normalfont\itshape}

\theoremsymbol{\ensuremath{\square}}
\newtheorem{proof}{Proof}

\usepackage{dsfont}

\newcommand{\R}{\mathds{R}}

\let\RE\Re
\let\Re=\undefined
\DeclareMathOperator{\Re}{\RE e}
\let\IM\Im
\let\Im=\undefined
\DeclareMathOperator{\Im}{\IM m}

\newcommand{\norm}[1]{\left\|#1\right\|}
\newcommand{\set}[1]{\left\{#1\right\}}
\newcommand{\dual}[2]{\left<#1,#2\right>}
\newcommand{\inner}[2]{\left(#1,#2\right)}
\newcommand{\abeta}{\rho}
\newcommand{\vary}{p}
\newcommand{\varlambda}{\mu}
\newcommand{\varx}{w}

\let\ii\i
\renewcommand{\i}{\mathrm i}

\newcommand{\wave}{\mathcal{W}}
\newcommand{\rhs}{f}

\begin{document}

\maketitle
\thispagestyle{empty}
\begin{center}
\parbox[t]{10em}{\footnotesize
\hspace*{-1ex}$^1$Faculty of Mathematics\\
University of Vienna\\
Oskar-Morgenstern-Platz 1\\
A-1090 Vienna, Austria}
\hfil
\parbox[t]{15em}{\footnotesize
\hspace*{-1ex}$^2$Institute of Computational Mathematics\\
Johannes Kepler University Linz\\
Altenbergerstr. 69\\
A-4040 Linz, Austria}
\hfil
\parbox[t]{16em}{\footnotesize
\hspace*{-1ex}$^3$Johann Radon Institute for Computational and Applied Mathematics (RICAM)\\
Altenbergerstr. 69\\
A-4040 Linz, Austria}
\end{center}
\begin{abstract}
 This paper is concerned with \emph{robust} preconditioning of \emph{wave equations constrained} linear 
 inverse problems from \emph{boundary} observation data. The main result of this paper is a concept for 
 regularization parameter robust preconditioning. Analogous concepts have been developed for 
 control problems based on \emph{elliptic} partial equations before.
\end{abstract}

\section{Introduction} \label{sec:introduction}
In this paper we reformulate inverse problems for linear \emph{hyperbolic} equations from indirect \emph{boundary} measurement data 
in the framework of optimal control. 
The optimality system of the resulting hyperbolic PDE-constrained optimization problem is a saddle point problem in 
an infinite dimensional Hilbert space setting. 
We prove well-posedness of the optimality system and provide a formulation that is \emph{regularization parameter robust}. 
In particular, we propose a preconditioner for the continuous system which renders the condition number of the preconditioned 
optimality system to be uniformly bounded with respect to the regularization parameter. 
Using a conformal finite element discretization, the discrete preconditioned system is robust with 
respect to the regularization \emph{and} discretization parameters. 
Analogous concepts have been developed for control problems based on \emph{elliptic} partial equations.
Saddle point formulations for inverse problems of the wave equation have been considered in \cite{CinMue15,CinMue16} before.
The main difference to these papers is that our approach involves an additional regularization term which allows us to 
avoid an ``observability hypothesis'' as stated in \cite{CinMue15,CinMue16}. A further difference is that for our choice of 
finite element discretization spaces, the discrete inf-sup condition of the associated saddle point problem is inherited from the continuous formulation 
(see \autoref{subsec:disc}) and does not need to be assumed.

As prime (numerical) test example we consider the inverse problem of Photoacoustic Tomography (PAT) \cite{WanWu07,Wan09}. 
For this inverse problem we provide a proof of concept of regularization parameter robust preconditioning. The concept is 
flexible and can be applied to generalized problems of PAT such as models taking into account attenuation or variable sound 
speed models. Moreover, at the current state of research we do not make an attempt to be competitive with existing highly 
developed and efficient algorithms in PAT (such as Fast Fourier methods based on backprojection algorithms, see for instance 
\cite{CoxTre10}), because our implementations are highly memory and run-time demanding since they require 
space-time solutions of the wave equation. 
For the sake of completeness we review the basic mathematical model of PAT and the variant, which is considered here: For 
photoacoustic imaging a specimen is illuminated by a short laser pulse. The absorbed electromagnetic energy 
creates an instantaneous heating of the probe which in turn induces an acoustic pressure wave caused by rapid thermal expansion. 
The goal of PAT is to reconstruct the electromagnetic absorption density, which is assumed to be proportional to 
the induced acoustic pressure (denoted by $u$ in the following). 
Mathematically, the propagating pressure wave (denoted by $y$) is typically modeled with the \emph{acoustic wave equation} in $\R^2, \R^3$:
Let $d=2,3$, then the pressure wave solves
\begin{equation}\label{eq:originalpat}
\begin{aligned}
\frac{1}{c^2(x)}y''(x,t) - \Delta y(x,t) &= 0 \quad \text{in}\quad \R^d\times(0,\infty),\\
y(x,0)=u(x),\quad y'(x,0)&=0  \quad \text{in}\quad \R^d.
\end{aligned}
\end{equation}
In PAT the acoustic pressure $y$ is measured over time on a curve $\Gamma$, which does not intersect the support of the specimen 
and lies entirely on one side of $\Gamma$ - with this 
term we have in mind for instance that the measurement devices surround the specimen or that 
the measurements lie on a part of a half plane, which bounds the specimen.
The mathematical problem of PAT consists in reconstructing
$u$ from the knowledge of $y$ on $\Gamma$ over time.
The inverse problem of PAT is therefore an inverse initial source problem for a hyperbolic partial differential equation.

In the forward modeling of PAT we slightly deviate from the standard model \autoref{eq:originalpat}, and consider the wave 
equation in a bounded domain $\Omega \subset \R^d$ over a finite time interval $(0,T)$ and enforce homogeneous Dirichlet boundary 
conditions on $\partial \Omega$.
The deviation from the standard model in PAT is done in order to work on a bounded computational domain.
If $\partial \Omega$ is relatively far away from the specimen of interest we might expect only slight deviations from the 
standard PAT model.
Therefore the PAT problem considered in this paper consists in estimating the absorption density $u : \Omega \to \R$ from boundary 
measurements $z_d$ of $y$ on $\Gamma \times (0,T)$, where $y$ is the solution of 
\begin{equation}\label{eq:patbounded}
\begin{aligned}
\frac{1}{c^2(x)}y''(x,t) - \Delta y(x,t) &= 0 \quad \text{in}\quad \Omega \times(0,T),\\
y(x,t) &=0 \quad \text{in} \quad \partial \Omega \times[0,T],\\
y(x,0)=u(x),\quad y'(x,0)&=0  \quad \text{in}\quad\Omega.
\end{aligned}
\end{equation}
We assume that $\Gamma$ is a closed curve completely contained in the open domain $\Omega$ (see \autoref{fig:L}). 
With respect to optimal control variants of the standard PAT model have been considered for instance also in \cite{ClaKli07}.
For the sake of simplicity, we shall illustrate our ideas here only for the constant sound speed coefficient case, 
the basic concept also applies to the variable sound speed case.

The organization of the paper is as follows. In \autoref{sec:ProblemStatement} we introduce some notation and recall well-known 
results about the wave equation which are needed in order to introduce our optimal control problem in a proper Hilbert 
space setting in \autoref{sec:mf}.
We prove well-posedness of the corresponding optimality system in a regularization parameter robust manner in 
\autoref{th:Brezziconstants}, which is the key element for the design of our robust preconditioner. 
In \autoref{sec:preconditioning} we propose and analyze a robust operator preconditioner for the continuous optimality 
system (see \autoref{th:contpcsystem}). 
We then show well-posedness of the discretized augmented optimality system for our particular choice of finite element 
discretization spaces.
Given such a stable discretization, the preconditioner for the discrete optimality system is derived from the continuous 
one in \autoref{th:discpcsystem}. 
Finally, in \autoref{sec:numerics}, we conclude our work with numerical experiments on the robustness of our preconditioner 
and the proof of concept of applicability for PAT.

\subsection*{Basic notation} 
All along this paper we will use the following notation:
\begin{notation}[Sets] \label{de:not}
 $\emptyset \neq \Omega \subset  \R^d$, $d \geq 1$, denotes an open bounded domain with piecewise Lipschitz boundary 
 $\partial\Omega$. Let $\emptyset \neq \Omega_s$ be an open subset with Lipschitz boundary $\partial \Omega_s$, which 
 is compactly supported in $\Omega$. Moreover, let $\Gamma \subset \partial\Omega_s$ be a measurable subset, see  \autoref{fig:L}.
 For $0<T<\infty$ we define $\mathcal{X}_T := \mathcal{X} \times (0,T)$ where $\mathcal{X} \in \set{\Omega,\partial \Omega_s, \partial\Omega, \Gamma}$.
 \begin{figure}[bht]
  \begin{center}
   \begin{tikzpicture}
    \filldraw[color=green, fill=green, very thick]
          (0,0.125\textwidth) -- (0.25\textwidth,0.125\textwidth) (0.25\textwidth,0.125\textwidth)--(0.25\textwidth,-0.125\textwidth)
          (0.25\textwidth,-0.125\textwidth) -- (0,-0.125\textwidth) (0,-0.125\textwidth) -- (0,0.125\textwidth);
    \filldraw[color=red!60, fill=red!5, thick]  (2.0,0.0) circle (0.06\textwidth);
    \filldraw [color=blue!60, fill=blue!5, thick] plot [smooth cycle, tension=2] coordinates {(1.5,0.2) (2.5,-0.1) (2,-0.6)};
    \draw(2.0,0.06\textwidth)[very thick] arc[radius = 0.06\textwidth, start angle= 90, end angle= 270];
 \draw (0.2\textwidth,0.08\textwidth) node[anchor=west] [color=green]{$\Omega$};
 \draw (0.055\textwidth,0.055\textwidth) node[anchor=west] [color=black]{$\Gamma$}; 
 \draw (2.0,0.5) node[anchor=west] [color=red]{$\Omega_s$};
 \draw (1.7,-0.1) node[anchor=west] [color=blue]{$u$};
\end{tikzpicture}
\end{center}
\caption{\label{fig:L} The wave equation is considered in $\Omega$, $\Gamma$ denotes the measurement set, and $\Omega_s$ 
contains the support of $u$, which is compactly supported in $\Omega_s$.}
\end{figure}
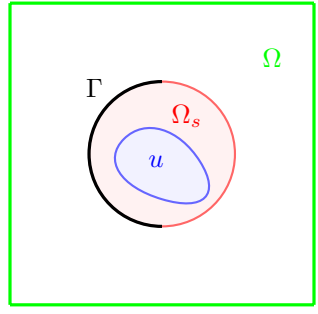
\end{notation}

\section{Weak solutions of the hyperbolic wave equation}\label{sec:ProblemStatement}
In the following we recall the concept of a weak solution of the wave equation. 
\begin{definition}[Weak solution \cite{Lio71}]\label{de:weaksolution} 
Let 
\begin{equation}\label{eq:W}
 W:=\set{ y \in L^2(0,T;H^1_0(\Omega)) : y'\in L^2(0,T;L^2(\Omega)),\, y''\in L^2(0,T;H^{-1}(\Omega)) }
\end{equation}
then the \emph{wave operator} is defined as 
\begin{equation} \label{eq:OperA}
 \begin{aligned}
  \wave: W &\rightarrow L^2(0,T;H^{-1}(\Omega)),\\
         y & \mapsto y'' - \Delta y. 
 \end{aligned}
\end{equation}
For given 
\begin{equation}\label{eq:X}
 (\rhs, y_0, y_1) \in D:= L^2(0,T;L^2(\Omega)) \times H^1_0(\Omega) \times L^2(\Omega),
\end{equation}
a function $y \in  W$ is called a \emph{weak solution} of the wave equation if it satisfies
\begin{equation}\label{eq:operatorwaveequation}
  \wave[y]= f \in L^2(0,T;L^2(\Omega)),\quad y(0)=y_0, \quad y'(0)=y_1.
\end{equation}
\end{definition}
Functions in $W$ satisfy
\begin{equation*}
 (y,y') \in C([0,T];L^2(\Omega))\times C([0,T];H^{-1}(\Omega)),
\end{equation*}
and thus, in particular, the initial conditions \autoref{eq:operatorwaveequation} are well-defined.
General results on existence and uniqueness of weak solutions of the wave equation from \cite{Lio71} and \cite{LioMag72a} are collected 
here for the readers convenience.
\begin{theorem}{\cite[Chapter 4, Theorem 1.1]{Lio71} and \cite[Chapter 3, Theorem 8.2]{LioMag72a}}\label{th:existenceanduniqueness}
 \begin{itemize}
  \item For every $(\rhs,y_0,y_1) \in D$ 
   there exists a unique weak solution $y\in W$ of \autoref{eq:operatorwaveequation}.
  \item The linear mapping $(\rhs,y_0,y_1) \mapsto (y,y')$ 
   is bounded from $D$ into $L^2(0,T;H^1_0(\Omega))\times L^2(0,T;L^2(\Omega))$ 
   and $C([0,T];H^1_0(\Omega))\times C([0,T];L^2(\Omega))$, respectively.
  \end{itemize}
\end{theorem}
Now, similar as in \cite{CinMue15, CinMue16}, we collect the set of possible solutions of the wave equation with 
inhomogeneities from the set $D$:
\begin{definition}\label{def:YY} Let $W$ be the set defined in \autoref{eq:W}. 
The \emph{set of possible solutions} of the wave operator is defined as the set 
\begin{equation} \label{eq:YY}
  Y:=\set{y \in W : 
  \exists (\rhs,y_0,y_1) \in D \text{ s.t. }
           \wave[y] = \rhs,\, y(0) = y_0,\, y'(0) = y_1}.
\end{equation}
\end{definition}
It is important for our further considerations, and somehow surprising, that the set $Y$ can be associated with a Hilbert-space topology:
\begin{theorem}[Hilbert-space $Y$]\label{eq:hspace} 
 The set $Y$ is a Hilbert space when associated with the inner product 
\begin{equation}\label{eq:inner_product}
 \inner{y}{\vary}_Y := 
 \inner{\wave[y]}{\wave[\vary]}_{L^2(0,T;L^2(\Omega))} + 
 \inner{y(0)}{\vary (0)}_{H^1_0(\Omega)} + 
 \inner{y'(0)}{\vary'(0)}_{L^2(\Omega)}.
 \end{equation}
\end{theorem}
\begin{proof}
\begin{enumerate}
 \item First, we note that from \autoref{th:existenceanduniqueness} it follows that the mapping 
       $(f,y_0,y_1)\mapsto y$ from $D$ to $Y$, where $y$ is the weak solution of \autoref{eq:operatorwaveequation}, is bijective with inverse $y\mapsto (\wave[y],y(0),y'(0))$.
 \item To see the completeness of $Y$ 
   let $(y_k)_k$ be a Cauchy sequence in $Y$. By the definition of the norm in $Y$ it follows that 
   $(\wave[y_k],y_k(0),y'_k(0))_k$ is a Cauchy sequence in $D$, and therefore possesses a limit in $D$, which we 
   denote by $(\rhs,y_0,y_1) $. According to the first item of this proof there exists $y \in Y$ which solves 
   $\wave[y]= \rhs$, $y(0)=y_0$, $y'(0)=y_1$. 
   It then follows that
   \begin{equation*}
    \norm{y_k-y}_Y^2=\norm{\wave[y_k]-f}^2_{L^2(0,T;L^2(\Omega))}+\norm{y_k(0)-y_0}^2_{H^1_0(\Omega)}+\norm{y'_k(0)-y_1}^2_{L^2(\Omega)}
    \rightarrow 0 \text{ as } k\rightarrow 0.
   \end{equation*}
   Thus the Cauchy sequence is converging and thus the space $Y$ is complete.
\end{enumerate}
\end{proof}

\section{The minimization functional} \label{sec:mf}
We consider now the problem of PAT on a bounded domain $\Omega$ as discussed in \autoref{sec:introduction}. For the sake of simplicity 
of the presentation we assume that the sound speed is constant one in $\Omega$. That is, the considered problem of photoacoustics 
(in operator notation) reads as follows 
\begin{problem}[PAT]\label{pr:weak}
 Given measurement data $z_d\in L^2(0,T;L^2(\Gamma))$. Find $(u,y) \in H_0^1(\Omega)\times Y$ such that $y = z_d$ on 
 $\Gamma \times (0,T)$, and such that 
 \begin{equation} \label{eq:PAT}
  \wave[y]= 0,\quad y(0)=u,\quad y'(0)=0.
 \end{equation}
\end{problem}
The solution of \autoref{pr:weak} appears to be unstable and thus we investigate a regularization technique 
consisting in calculating, for some fixed $\alpha > 0$, a minimizer of the cost functional 
\begin{equation}\label{eq:minimizingfunctional2}
 \begin{aligned}
 \mathcal{T}_\alpha: Y &\rightarrow \R,\\
 y &\mapsto \frac{1}{2}\norm{y-z_d}^2_{L^2(0,T;L^2(\Gamma))} + \frac{\alpha}{2}\norm{y(0)}^2_{H^1_0(\Omega)},
 \end{aligned}
\end{equation}
subject to the constraint
\begin{equation} \label{eq:ynu}
 y \in Y_0:=\set{y \in Y : \wave[y]=0,\, y'(0)=0}.
\end{equation}
In the following we analyze the functional $\mathcal{T}_\alpha$. The first result consists in proving that 
the trace of $y$ is well-defined such that the residual term $\norm{y-z_d}^2_{L^2(0,T;L^2(\Gamma))}$ is well-defined.
\begin{lemma}
 There exists a positive constant $C_{\text{obs}}:=C_{\text{obs}}(\Omega,\Omega_s,T)$ such that
 \begin{equation}\label{eq:continuityobservation}
  \norm{y}_{L^2(0,T;L^2(\Gamma))}\leq C_{\text{obs}}\norm{y}_Y\qquad\text{for all}\quad y\in Y.
 \end{equation}
\end{lemma}
\begin{proof}
 From \autoref{th:existenceanduniqueness} it follows that $Y$ as definedin \autoref{eq:YY} can be represented as
 \begin{equation}\label{eq:imagesolutionoperator}
 Y =\set{ y \in W : (y,y') \in C([0,T];H^1_0(\Omega))\times C([0,T];L^2(\Omega)),\, \wave[y]\in L^2(0,T;L^2(\Omega))},
 \end{equation}
 and moreover the norms
 \begin{equation*}
  \norm{y}_C:=\norm{y}_{C([0,T];H^1_0(\Omega))} + \norm{y'}_{C([0,T];L^2(\Omega))} + \norm{\wave[y]}_{L^2(0,T;L^2(\Omega))}
 \end{equation*}
 and $\norm{\cdot}_Y$ are equivalent. This together with the trace theorem, shows that
 \begin{align*}
   \norm{y}_{L^2(0,T;L^2(\Gamma))}\leq  \norm{y}_{L^2(0,T;L^2(\partial \Omega_s))}&\leq C(\Omega_s)\norm{y}_{L^2(0,T;H^1(\Omega_s))}\\
   &\leq C(\Omega,\Omega_s)\norm{y}_{L^2(0,T;H^1_0(\Omega))}\leq C(\Omega,\Omega_s,T) \norm{y}_{C([0,T];H^1_0(\Omega))},
 \end{align*}
 which gives the assertion.
\end{proof}
In order to solve the constrained minimization problem we will reformulate it as a saddle point problem.

\subsection{The saddle point problem}
Set 
\begin{equation} \label{eq:lambda}
 \Lambda := L^2(0,T;L^2(\Omega)) \times L^2(\Omega)
\end{equation}
equipped with the standard product norm
$\norm{\lambda}^2_\Lambda=\norm{\lambda_1}_{L^2(0,T;L^2(\Omega))}^2+\norm{\lambda_2}_{L^2(\Omega)}^2.$

\begin{definition}[(Augmented) Lagrangian]
 Let $\abeta \geq 0$. The \emph{Lagrangian} ($\abeta=0$), \emph{augmented Lagrangian} ($\abeta > 0$), respectively, 
associated to the minimization functional $\mathcal{T}_\alpha$ from \autoref{eq:minimizingfunctional2} reads as follows
\begin{equation}\label{eq:Lagrangian1}
 \begin{aligned}
  \mathcal{L}_\abeta: Y \times \Lambda &\rightarrow \R, \\
    (y,\lambda) &\mapsto \frac{1}{2}a_{\alpha,\abeta}(y,y) + b(y,\lambda)-l(y),
 \end{aligned}
\end{equation}
where the bilinear forms $a_{\alpha,\abeta}:Y \times Y \rightarrow \R$ and $b:Y \times \Lambda \rightarrow \R$ and the linear form 
$l:Y\rightarrow\R$ are defined by
\begin{equation} \label{eq:bdef1}
 \begin{aligned}
    a_{\alpha,\abeta}(y,\vary)&=\inner{y}{\vary}_{L^2(0,T;L^2(\Gamma))} + \alpha\inner{y(0)}{\vary (0)}_{H^1_0(\Omega)} \\
        & \qquad + 
         \abeta\left[\inner{\wave[y]}{\wave[\vary]}_{L^2(0,T;L^2(\Omega))} + \inner{y'(0)}{\vary'(0)}_{L^2(\Omega)}\right] 
         \text{ for all } y,\vary \in Y,\\
    b(y,\lambda)&=\inner{\wave[y]}{\lambda_1}_{L^2(0,T;L^2(\Omega))} + \inner{y'(0)}{\lambda_2}_{L^2(\Omega)}\text{ for all } y \in Y, \lambda \in \Lambda,\\ 
    l(y)&=\inner{z_d}{y}_{L^2(0,T;L^2(\Gamma))} \text{ for all }y \in Y.
 \end{aligned}
\end{equation}
\end{definition}

The first order optimality conditions of the Lagrangian $\mathcal{L}_\abeta$, 
\begin{equation*}
 \frac{\partial\mathcal{L}_\abeta}{\partial y}(y,\lambda) = \frac{\partial\mathcal{L}_\abeta}{\partial \lambda}(y,\lambda) = 0,
\end{equation*}
can be expressed as a saddle point problem, consisting in finding $(y,\lambda) \in Y \times \Lambda$ satisfying
\begin{equation}\label{eq:saddlepointproblem1}
 \begin{aligned}
  a_{\alpha,\abeta}(y,\vary)+b(\vary ,\lambda) &= l(\vary) & \text{ for all } &\vary \in Y,\\
  b(y,\varlambda )& = 0 &\text{ for all } &\varlambda \in\Lambda.
  \end{aligned}
\end{equation}
To prove existence and uniqueness of a solution of \autoref{eq:saddlepointproblem1} we need to guarantee four Brezzi-conditions 
(see for instance \cite{BofBreFor13}) in an appropriate Hilbert space setting on $Y$ (as defined in \autoref{eq:YY}). 
In particular we investigate $Y$ together with a parameter dependent family of functionals:
For $\abeta > 0$ we introduce 
\begin{equation}\label{eq:Ynorm}
\norm{y}_{Y_{\alpha,\abeta}}^2 := a_{\alpha,\abeta}(y,y) = 
\norm{y}^2_{L^2(0,T;L^2(\Gamma))} + \alpha\norm{y(0)}^2_{H^1_0(\Omega)} + \abeta\left[\norm{\wave[y]}^2_{L^2(0,T;L^2(\Omega))}+\norm{y'(0)}^2_{L^2(\Omega)}\right].
\end{equation} 
The next lemma guarantees that $Y$ associated with the bilinear forms $\inner{\cdot}{\cdot}_{Y_{\alpha,\abeta}}$, induced by 
$\norm{y}_{Y_{\alpha,\abeta}}$ is indeed a Hilbert-space.
\begin{lemma}
For every $\alpha, \abeta > 0$, $\norm{\cdot}_{Y_{\alpha,\abeta}}$ is a norm on $Y$.
\end{lemma}
\begin{proof}
 According to \autoref{eq:hspace} the mapping 
 $$y\mapsto \left(\alpha\norm{y(0)}^2_{H^1_0(\Omega)} + \abeta\left[\norm{\wave[y]}^2_{L^2(0,T;L^2(\Omega))}+\norm{y'(0)}^2_{L^2(\Omega)}\right]\right)^{1/2}$$ 
 is an equivalent norm to $\norm{\cdot}_Y$. The assertion then follows from \autoref{eq:continuityobservation}.
\end{proof}
Moreover, for verifying the Brezzi-conditions we also need the following lemma:
\begin{lemma}
The kernel of the bilinear form $b$, $\mathcal{N}(b):=\set{y\in Y : b(y,\lambda)=0 \text{ for all } \lambda \in \Lambda}$,  satisfies
\begin{equation} \label{eq:null}
\mathcal{N}(b)=Y_0,
\end{equation}
where $Y_0$ is defined in \autoref{eq:ynu}.
\end{lemma}
\begin{proof}
 Let $y\in\mathcal{N}(b)$, according to the definition of $Y$, \autoref{eq:YY}, 
 $(\wave[y],y'(0))\in L^2(0,T;L^2(\Omega)) \times L^2(\Omega) = \Lambda$ (see \autoref{eq:YY}) it follows from 
 the definition of $b$, \autoref{eq:bdef1}, that
 \begin{equation*}
  b(y,(\wave[y],y'(0)))=\norm{\wave[y]}^2_{L^2(0,T;L^2(\Omega))}+\norm{y'(0)}^2_{L^2(\Omega)}=0,
 \end{equation*}
that is $y\in Y_0$. The other inclusion is trivial. 
\end{proof}

\begin{theorem}[Brezzi conditions]\label{th:Brezziconstants}
Let $\alpha, \abeta > 0$.
\begin{description}
 \item[The 1st Brezzi condition] (boundedness of $a_{\alpha,\abeta}$) holds,
  \begin{equation}\label{eq:ca}
   a_{\alpha,0}(y,\vary)\leq\norm{y}_{Y_{\alpha,\abeta}}\norm{\vary}_{Y_{\alpha,\abeta}}\text{ and } 
   a_{\alpha,\abeta}(y,\vary)\leq\norm{y}_{Y_{\alpha,\abeta}}\norm{\vary}_{Y_{\alpha,\abeta}}\text{ for all } y,\vary \in Y.
  \end{equation}
 \item[The 2nd Brezzi condition] (coercivity of $a_{\alpha,\abeta}$ on the kernel of $b$) holds,
   \begin{equation}\label{eq:kernel}
   a_{\alpha,0}(y,y)=a_{\alpha,\abeta}(y,y)=\norm{y}^2_{Y_{\alpha,\abeta}} \text{ for all } y\in \mathcal{N}(b).
   \end{equation}
   Moreover,
   \begin{equation*}
    a_{\alpha,\abeta}(y,y)=\norm{y}^2_{Y_{\alpha,\abeta}} \text{ for all } y\in Y.
   \end{equation*}
 \item[The 3rd Brezzi condition] (boundedness of $b$) holds,
  \begin{equation}\label{eq:cb}
   b(y,\lambda)\leq c_b\norm{y}_{Y_{\alpha,\abeta}}\norm{\lambda}_\Lambda \text{ for all } y\in Y,\,\lambda\in\Lambda\text{ with }c_b=\frac{1}{\sqrt{\abeta}}.
   \end{equation}
 \item[The 4th Brezzi condition] ($b$ satisfies the inf-sup condition) holds,
   \begin{equation}\label{eq:brezzicoercivity}
    \sup_{0 \neq y \in Y} \frac{b(y,\lambda)}{\norm{y}_{Y_{\alpha,\abeta}}} \geq k_0\norm{\lambda}_\Lambda \text{ for all } 
    \lambda\in\Lambda\text{ with } k_0=\frac{1}{\sqrt{C_{\text{obs}}^2+\abeta}},
   \end{equation}
   where $C_{\text{obs}}=C_{\text{obs}}(\Omega,\Omega_s,T)$ is the constant from \autoref{eq:continuityobservation}.
\end{description}
\end{theorem}
\begin{proof}
\begin{itemize} 
 \item To prove the 1st Brezzi condition we estimate with Cauchy-Schwarz inequality as follows: Let $y$, $\vary \in Y$, then 
  \begin{align*}
   a_{\alpha,0}(y,\vary) 
   &=    \inner{\left(y,\sqrt{\alpha}y(0)\right)}{\left(\vary,\sqrt{\alpha}\vary(0)\right)}_{L^2(0,T;L^2(\Gamma))\times H^1_0(\Omega)}\\
   &\leq \norm{\left(y,\sqrt{\alpha}y(0)\right)}_{L^2(0,T;L^2(\Gamma))\times H^1_0(\Omega)}
         \norm{\left(\vary,\sqrt{\alpha}\vary(0)\right)}_{L^2(0,T;L^2(\Gamma))\times H^1_0(\Omega)}\\
   &\leq \norm{y}_{Y_{\alpha,\abeta}}\norm{\vary}_{Y_{\alpha,\abeta}}.
  \end{align*}
  The second inequality follows directly from applying the Cauchy-Schwarz inequality to the inner product $\inner{\cdot}{\cdot}_{Y_{\alpha,\abeta}}=a_{\alpha,\abeta}(\cdot,\cdot)$.
 \item In an analogous manner one shows for the 3rd Brezzi condition that for all $y\in Y$, $\lambda\in \Lambda$
 \begin{equation*}
  b(y,\lambda) = \frac{1}{\sqrt{\abeta}}\inner{\sqrt{\abeta}\left(\wave[y],y'(0)\right)}{\left(\lambda_1,\lambda_2\right)}_{L^2(0,T;L^2(\Omega))\times L^2(\Omega)} \leq \frac{1}{\sqrt{\abeta}}\norm{y}_{Y_{\alpha,\abeta}}\norm{\lambda}_\Lambda.
 \end{equation*}
 \item For proving the 2nd Brezzi condition we note that for all $y\in \mathcal{N}(b)=Y_0$
   \begin{equation*}
    a_{\alpha,0}(y,y)=\norm{y}^2_{L^2(0,T;L^2(\Gamma))}+\alpha\norm{y(0)}^2_{H^1_0(\Omega)}= a_{\alpha,\abeta}(y,y)= \norm{y}^2_{\alpha,\abeta},
   \end{equation*}
   which gives the assertion. Trivially, $a_{\alpha,\abeta}(y,y) = \norm{y}^2_{\alpha,\abeta}$ for all $y\in Y$ by the definition of $\norm{\cdot}_{Y_{\alpha,\abeta}}$ in \autoref{eq:Ynorm}.
 \item To prove the 4th Brezzi condition let 
   $0 \neq \lambda=(\lambda_1,\lambda_2) \in \Lambda$ be arbitrary and choose $\hat{y}\in Y$ 
   such that 
   \begin{equation*}
    \wave[\hat{y}]=\lambda_1,\qquad \hat{y}(0)=0,\qquad \hat{y}'(0)=\lambda_2.
   \end{equation*}
Then
\begin{equation*}
  \sup_{0\neq y\in Y}\frac{b(y,\lambda)}{\norm{y}_{Y_{\alpha,\abeta}}}\geq
 \frac{\norm{\lambda_1}^2_{L^2(0,T;L^2(\Omega))}+\norm{\lambda_2}^2_{L^2(\Omega)}}{\sqrt{\norm{\hat{y}}^2_{L^2(0,T;L^2(\Gamma))}+\abeta\left(\norm{\lambda_1}^2_{L^2(0,T;L^2(\Omega))}+\norm{\lambda_2}^2_{L^2(\Omega)}\right)}}\geq\frac{1}{\sqrt{C_{\text{obs}}^2+\abeta}}\norm{\lambda}_\Lambda,
 \end{equation*}
where we used  \autoref{eq:continuityobservation} in the last step.
\end{itemize}
\end{proof}
Existence and uniqueness of the saddle point problem \autoref{eq:saddlepointproblem1} follows from \cite[Theorem 4.2.3]{BofBreFor13} 
and \autoref{th:Brezziconstants}. Moreover, note that $a_{\alpha,0}(y,y)=a_{\alpha,\abeta}(y,y)$ for the solution $(y,\lambda)$ to \autoref{eq:saddlepointproblem1}, whence the Lagrangians $\mathcal{L}_{\abeta\geq0}$ share the same saddle point.
\begin{corollary}
 For every $\alpha>0$, $\abeta\geq0$ there exists a unique pair $(y_\alpha,\lambda)\in Y\times \Lambda$ satisfying the mixed variational problem \autoref{eq:saddlepointproblem1}. 
 As a consequence, the function $y_\alpha$ is the unique minimizer of the constrained optimization problem \autoref{eq:minimizingfunctional2}.
\end{corollary}
So far we have seen that for every $\alpha>0$ there exists a unique pair $(y,\lambda)\in Y\times \Lambda$ solving the saddle point problem 
\autoref{eq:saddlepointproblem1}. It is due to the regularization term in the minimization functional \autoref{eq:minimizingfunctional2} that the coercivity 
of $a_{\alpha,\abeta}$ $(\abeta\geq0)$ on the kernel of $b$ holds. In the work of C\^indea and M\"unch \cite{CinMue15, CinMue16} the authors consider the minimization of the data 
discrepancy only, without an additional regularization term, and thus, in order to establish coercivity, they needed to assume an additional
\emph{observability inequality} which guarantees that the measurements yield a norm on the kernel of their state equation. We can avoid such an assumption.

\subsection{The saddle point problem in operator notation}\label{subsec:operatornotation}
We recall that from \autoref{th:Brezziconstants} it follows that $a_{\alpha,0}$, $a_{\alpha,\abeta}$ and $b$ are bounded bilinear 
forms on $(Y,\norm{\cdot}_{Y_{\alpha,\abeta}})$, resp. $(\Lambda,\norm{\cdot}_\Lambda)$, for $\alpha,\abeta>0$. In the following we will denote the set $Y$ as the Hilbert space $(Y,\norm{\cdot}_{Y_{\alpha,\abeta}})$ and similarly $\Lambda$ for $(\Lambda,\norm{\cdot}_\Lambda)$.

Thus for every $\alpha > 0$ and $\rho \geq 0$ there exists a continuous operator $A_{\alpha,\abeta}:Y\rightarrow Y'$ such that 
\begin{equation*}
 \dual{A_{\alpha,\abeta} y}{\vary}_{Y'\times Y}=a_{\alpha,\abeta}(y,\vary) \text{ for all } y,\vary \in Y.
\end{equation*}
In an analogous manner we see that there exists a bounded operator $B:Y\rightarrow \Lambda'$ satisfying
\begin{equation*}
 \dual{By}{\lambda}_{\Lambda'\times \Lambda}=b(y,\lambda) \text{ for all }  y,\vary \in Y,\,\lambda\in\Lambda.
\end{equation*}
For $B:Y\rightarrow \Lambda'$, its adjoint $B':\Lambda\rightarrow Y'$ is given by
\begin{equation*}
 \dual{B' \lambda}{y}_{Y'\times Y}=\dual{By}{\lambda}_{\Lambda'\times \Lambda}=b(y,\lambda) \text{ for all }  y\in Y,\, \lambda\in\Lambda.
\end{equation*}
Using this operator notation, the saddle point problem, \autoref{eq:saddlepointproblem1}, is equivalent to 
\begin{equation}\label{eq:sppof1}  
\mathcal{A}_{\alpha,\abeta} \begin{pmatrix} 
                           y\\
                           \lambda\end{pmatrix}:=
 \begin{pmatrix}
  A_{\alpha,\abeta} & B'\\
  B & 0
 \end{pmatrix}
 \begin{pmatrix}
 y\\
 \lambda
 \end{pmatrix}
 =
 \begin{pmatrix}
 l\\
 0
 \end{pmatrix}.
\end{equation}
The following corollary is an immediate consequence of \autoref{th:Brezziconstants}.
\begin{corollary}\label{co:isomorphism}
Let $\alpha>0$. For $\abeta>0$ the linear operator $\mathcal{A}_{\alpha,\abeta}: Y \times\Lambda\rightarrow Y'\times \Lambda'$ 
defined in \autoref{eq:sppof1} is a self-adjoint isomorphism. Furthermore, for the Hilbert space $X:=Y\times \Lambda$ endowed with the inner product
\begin{equation}\label{eq:inner_productX}
 \inner{x}{\varx}_{X_{\alpha,\abeta}}:=\inner{y}{\vary}_{Y_{\alpha,\abeta}}+\inner{\lambda}{\varlambda}_\Lambda\quad 
 \text{ for all } x=(y,\lambda), \varx=(\vary,\varlambda)\in X,
\end{equation}
it follows that 
\begin{equation}\label{eq:operatornorms}
\norm{\mathcal{A}_{\alpha,\abeta}}_{\mathcal{L}(X,X')}\leq \overline{c}\quad\text{ and } \norm{\mathcal{A}_{\alpha,\abeta}^{-1}}_{\mathcal{L}(X',X)}\leq \frac{1}{\underline{c}} ,
\end{equation}
where the constants $\overline{c}=\overline{c}(\abeta)$ and $\underline{c}=\underline{c}(\abeta)$ are positive and independent of $\alpha$. Here, as usual, we identify the dual $X'$ of $X = Y\times\Lambda$ with $Y' \times \Lambda'$.

Similarly, the linear operator $\mathcal{A}_{\alpha,0}: Y \times\Lambda\rightarrow Y'\times \Lambda'$ 
defined in \autoref{eq:sppof1} is a self-adjoint isomorphism and \autoref{eq:operatornorms} holds 
for arbitrary $\abeta>0$.
\end{corollary}

\section{Robust preconditioning}\label{sec:preconditioning}
Since we are dealing with a space-time domain $\Omega_T$, a discretization of \autoref{eq:sppof1} will lead to a 
very large linear system of equations. For the efficient solution we will use a preconditioned minimum residual (MINRES) 
method. The regularization parameter $\alpha>0$ enters in \autoref{eq:sppof1} via the bilinear form 
$a_{\alpha,\abeta}$ ($\abeta \geq 0$). Thus the solution of the (discretized) system depends on $\alpha$. 
It is our goal to obtain $\alpha$-independent convergence of an appropriately designed preconditioned MINRES method
for the discretized system. Such a preconditioner for the discretized system is derived from an \emph{operator preconditioner} 
for the continuous system \autoref{eq:sppof1}.

\subsection{Operator preconditioning}
Since the operator $\mathcal{A}_{\alpha,\abeta}$, defined in \autoref{eq:sppof1} is not a self-mapping, 
a MINRES method \emph{cannot} be applied to the system \autoref{eq:sppof1}. 
This is remedied by complementing it with an isomorphic operator such that the composition is a self-mapping. 
This complementation is refered to as \emph{preconditioning}. For an in-depth discussion on this topic we refer 
to \cite{Hip06, MarWin10}.
\begin{theorem}\label{th:contpcsystem}
 Let $\alpha>0$. For $\abeta>0$ the Hilbert space $X= Y\times\Lambda$ equipped with the inner product $\inner{\cdot}{\cdot}_{X_{\alpha,\abeta}}$ from \autoref{eq:inner_productX}, 
 let $\mathcal{P}_{\alpha,\abeta}:X\rightarrow X'$ be defined by 
 \begin{equation} \label{eq:xar}
  \dual{\mathcal{P}_{\alpha,\abeta} x}{\varx}_{X'\times X} := \inner{x}{\varx}_{X_{\alpha,\abeta}} 
  \quad \text{ for all } x,\varx\in X.
 \end{equation}
 Then the operator $\mathcal{P}_{\alpha,\abeta}^{-1}\mathcal{A}_{\alpha,\abeta}:X\rightarrow X$ is a self-adjoint isomorphism with respect to the topology
 induced by the inner product $\inner{\cdot}{\cdot}_{X_{\alpha,\abeta}}$.
 
 Furthermore, the condition number of the preconditioned systems satisfies
 \begin{equation}\label{eq:precon}
  \kappa(\mathcal{P}_{\alpha,\abeta}^{-1}\mathcal{A}_{\alpha,\abeta}) = 
  \norm{\mathcal{A}_{\alpha,\abeta}}_{\mathcal{L}(X,X')} \norm{\mathcal{A}_{\alpha,\abeta}^{-1}}_{\mathcal{L}(X',X)}
 \end{equation}
and is bounded uniformly in $\alpha>0$.

Similarly, for the linear operator $\mathcal{A}_{\alpha,0}: Y \times\Lambda\rightarrow Y'\times \Lambda'$ 
defined in \autoref{eq:sppof1} the operator $\mathcal{P}_{\alpha,\abeta}^{-1}\mathcal{A}_{\alpha,0}:X\rightarrow X$ is a self-adjoint isomorphism with respect to the topology
induced by the inner product $\inner{\cdot}{\cdot}_{X_{\alpha,\abeta}}$ and \autoref{eq:precon} holds for arbitrary $\abeta>0$.
\end{theorem}
\begin{proof}
 The operator $\mathcal{P}_{\alpha,\abeta}: X\rightarrow X'$ is the inverse of the Riesz representation operator $J_{X'}:X'\rightarrow X$ 
 (see for instance \cite{Rud74}) when $X$ is associated with the topology induced by $\norm{\cdot}_{X_{\alpha,\abeta}}$ as introduced in 
 \autoref{eq:xar}. 
 
 Then from \autoref{co:isomorphism} it follows that the composition $\mathcal{P}_{\alpha,\abeta}^{-1}\mathcal{A}_{\alpha,\abeta}:X\rightarrow X$ is a linear isomorphism on 
 $X$ which is self-adjoint with respect to the inner product on $X$,
 \begin{equation*}
  \inner{\mathcal{P}_{\alpha,\abeta}^{-1}\mathcal{A}_{\alpha,\abeta} x }{\varx}_{X_{\alpha,\abeta}} = \dual{\mathcal{A}_{\alpha,\abeta}  x }{\varx}_{X'\times X} =  \dual{\mathcal{A}_{\alpha,\abeta} \varx}{x}_{X'\times X} = \inner{\mathcal{P}_{\alpha,\abeta}^{-1}\mathcal{A}_{\alpha,\abeta} \varx}{x}_{X_{\alpha,\abeta}},
 \end{equation*}
which shows the first claim.

For the second part note that
\begin{equation*}
 \norm{\mathcal{P}_{\alpha,\abeta}^{-1}\mathcal{A}_{\alpha,\abeta}}_{\mathcal{L}(X,X)} = \norm{\mathcal{A}_{\alpha,\abeta}}_{\mathcal{L}(X,X')} \quad \text{and} \quad \norm{\left(\mathcal{P}_{\alpha,\abeta}^{-1}\mathcal{A}_{\alpha,\abeta}\right)^{-1}}_{\mathcal{L}(X,X)} = \norm{\mathcal{A}^{-1}_{\alpha,\abeta}}_{\mathcal{L}(X',X)}.
\end{equation*}
The assertion then follows from the definition of the condition number and \autoref{eq:operatornorms}.

The same considerations apply for the operator $\mathcal{A}_{\alpha,0}$.
\end{proof}
In the following we use $\mathcal{P}_{\alpha,\abeta}$ as a preconditioner and as a consequence the minimum residual method can be applied to 
the preconditioned system of \autoref{eq:sppof1}:
 \begin{equation} \label{eq:sppof2}
  \mathcal{P}_{\alpha,\abeta}^{-1}\mathcal{A}_{\alpha,\abeta} 
  \begin{pmatrix}y\\\lambda\end{pmatrix} =\mathcal{P}_{\alpha,\abeta}^{-1}\begin{pmatrix}l\\0\end{pmatrix}.
 \end{equation}

\subsection{Stable discretization}\label{subsec:disc}
The well-posedness of the discretized version of \autoref{eq:sppof1} is characterized by the \emph{discrete} Brezzi conditions. For a pair of conforming discretization spaces $Y_h\subset Y$ and $\Lambda_h\subset \Lambda$, the bilinear forms $a_{\alpha,0}$, $a_{\alpha,\abeta}$ and $b$, restricted to the subspaces $Y_h\times Y_h$, $Y_h\times\Lambda_h$ respectively, satisfy the 1st and 3rd Brezzi condition from \autoref{th:Brezziconstants} automatically for every $\alpha,\abeta>0$. The 2nd Brezzi condition is in general only satisfied for $a_{\alpha,\abeta}$ for $\abeta>0$. We will therefore only consider the discretization of the optimality system for the \emph{augmented} Lagrangian $\mathcal{L}_\abeta$ ($\abeta>0$), see \autoref{eq:Lagrangian1}.

It remains to verify the 4th Brezzi condition. 
\begin{lemma}\label{le:discinfsup}
Let $Y_h\subset Y$ be a conforming discretization space and set
\begin{equation}
  \label{eq:LamDisc}
\Lambda_h :=\left\lbrace \lambda_h = (\wave[y_h], y_h'(0)) : y_h\in Y_{h},\, y_h(0) = 0 \right\rbrace.
\end{equation}
Then for every $\alpha,\abeta>0$, $\Lambda_h\subset\Lambda$ and the bilinear form $b$ satisfy the 4th Brezzi condition,
\begin{equation*}
 \sup_{0 \neq y_h \in Y_h} \frac{b(y_h,\lambda_h)}{\norm{y_h}_{Y_{\alpha,\abeta}}} \geq k_0\norm{\lambda_h}_{\Lambda_h} \text{ for all } 
 \lambda_h\in\Lambda\text{ with } k_0=\frac{1}{\sqrt{C_{\text{obs}}^2+\abeta}},
\end{equation*}
where $C_{\text{obs}}=C_{\text{obs}}(\Omega,\Omega_s,T)$ is the constant from \autoref{eq:continuityobservation}.
\end{lemma}
\begin{proof}
Since $Y_h\subset Y$ it follows that $(\wave[y_h], y_h'(0))\in L^2(0,T;L^2(\Omega))\times L^2(\Omega)=\Lambda$ for every $y_h\in Y_h$, whence $\Lambda_h\subset \Lambda$.

The proof of the second assertion follows exactly the same lines as the last part of the proof of \autoref{th:Brezziconstants} with $(Y,\Lambda,y,\lambda,\hat{y})$ replaced by $(Y_h,\Lambda_h,y_h,\lambda_h,\hat{y}_h)$. Here $\hat{y}_h\in Y_h$ is chosen such that
 \begin{equation*}
\wave[\hat{y}_h]=\lambda_{h,1},\qquad \hat{y}_h(0)=0,\qquad \hat{y}_h'(0)=\lambda_{h,2},
\end{equation*}
which exists by construction of $\Lambda_h$.
\end{proof}

With $A_{\alpha,\abeta,h}:Y_h\rightarrow Y_h'$ and $B_h:Y_h\rightarrow\Lambda_h'$ defined by 
\begin{equation}\label{eq:discoperators}
 \dual{A_{\alpha,\abeta,h}y_h}{\vary _h}_{Y_h'\times Y_h}=a_{\alpha,\abeta}(y_h,\vary _h),\qquad
 \dual{B_h y_h}{\lambda_h}_{\Lambda_h'\times \Lambda_h}=b(y_h,\lambda_h),
\end{equation}
and $l_h=l|_{Y_h}$, the discretized saddle point problem which is considered in the following is given by
\begin{equation}\label{eq:discsaddlepointproblemoperatorform}
  \mathcal{A}_{\alpha,\abeta,h}\begin{pmatrix}y_h\\\lambda_h\end{pmatrix} := \begin{pmatrix} A_{\alpha,\abeta,h} & 
  B_h'\\B_h&0\end{pmatrix}\begin{pmatrix}y_h\\\lambda_h\end{pmatrix}=\begin{pmatrix}l_h\\0\end{pmatrix}.
\end{equation}
Since the discrete Brezzi conditions are satisfied with the same constants as for the continuous formulation, \autoref{co:isomorphism} carries over to the 
discretized formulation in  \autoref{eq:discsaddlepointproblemoperatorform}.
\begin{corollary}\label{co:discreteisomorphism}
Let $\alpha,\abeta>0$. For a conforming discretization space $Y_h\subset Y$ and $\Lambda_h\subset\Lambda$ given by \autoref{eq:LamDisc}, the linear operator 
$\mathcal{A}_{\alpha,\abeta,h}:Y_h\times\Lambda_h\rightarrow Y_h'\times\Lambda_h'$ defined in \autoref{eq:discsaddlepointproblemoperatorform} is a symmetric 
isomorphism. Furthermore, for the discretization space $X_h=Y_h\times\Lambda_h$ endowed with the inner product $\inner{\cdot}{\cdot}_{X_{\alpha,\abeta}}$ from 
\autoref{eq:inner_productX}
it follows that
\begin{equation}\label{eq:operatornorms_discrete}
 \norm{\mathcal{A}_{\alpha,\abeta,h}}_{\mathcal{L}(X_h,X_h')}\leq \overline{c}\text{ and } \norm{\mathcal{A}_{\alpha,\abeta,h}^{-1}}_{\mathcal{L}(X_h',X_h)}\leq 
 \frac{1}{\underline{c}} ,
\end{equation}
where the constants $\overline{c}=\overline{c}(\abeta)$ and $\underline{c}=\underline{c}(\abeta)$, are positive and independent of $\alpha$ 
and do not depend on the choice of $Y_h$.
\end{corollary}
\begin{theorem}\label{th:discpcsystem}
 Let $\alpha,\abeta>0$ and assume that $Y_h\subset Y$ is a conforming discretization space and let $\Lambda_h\subset\Lambda$ be given by \autoref{eq:LamDisc}. Let 
 $\mathcal{P}_{\alpha,\abeta,h}:Y_h\times\Lambda_h=X_h\rightarrow X_h'$ be the matrix representation associated to the inner product $\inner{\cdot}
 {\cdot}_{X_{\alpha,\abeta}}$,
 \begin{equation*}
  \dual{\mathcal{P}_{\alpha,\abeta,h}x_h}{\varx _h}_{X_h'\times X_h}=\inner{x_h}{\varx_h}_{X_{\alpha,\abeta}}
  ,\quad x_h\,,\varx_h\in X_h.
 \end{equation*}
 Then the operator $\mathcal{P}_{\alpha,\abeta,h}^{-1}\mathcal{A}_{\alpha,\abeta,h}:X_h\rightarrow X_h$ is a symmetric isomorphism with respect to the topology 
 induced by the inner product $\inner{\cdot}{\cdot}_{X_{\alpha,\abeta}}$.
 
  Furthermore, the condition number of the preconditioned systems satisfies
 \begin{equation}\label{eq:precon2}
  \kappa(\mathcal{P}_{\alpha,\abeta,h}^{-1}\mathcal{A}_{\alpha,\abeta,h}) = 
  \norm{\mathcal{A}_{\alpha,\abeta,h}}_{\mathcal{L}(X_h,X_h')} \norm{\mathcal{A}_{\alpha,\abeta,h}^{-1}}_{\mathcal{L}(X_h,X_h')}
 \end{equation}
and is bounded uniformly in $\alpha>0$. Moreover, it does not depend on the choice of $Y_h$.
 
\end{theorem}
This shows that our preconditioner $\mathcal{P}_{\alpha,\abeta,h}$ is \emph{robust with respect to $\alpha$ and the discretization}.
\begin{proof}
 The theorem follows from \autoref{co:discreteisomorphism} and is proven along the lines of the proof of \autoref{th:contpcsystem} restricted to the subspace $X_h=Y_h\times\Lambda_h$.
\end{proof}

In the following we use $\mathcal{P}_{\alpha,\abeta,h}^{-1}$ as a preconditioner and as a consequence the minimum residual method can be applied to 
the preconditioned discrete system of \autoref{eq:discsaddlepointproblemoperatorform}:
\begin{equation}\label{eq:sppof4}
 \mathcal{P}_{\alpha,\abeta,h}^{-1}\mathcal{A}_{\alpha,\abeta,h}\begin{pmatrix}y_h\\\lambda_h\end{pmatrix}=\mathcal{P}_{\alpha,\abeta,h}^{-1}\mathcal{A}_{\alpha,\abeta,h}\begin{pmatrix}l_h\\0\end{pmatrix}.
\end{equation}

\begin{remark}\label{re:rhorobustness}
 Note that albeit the robustness with respect to $\alpha$ and the discretization, the condition number of the preconditioned discrete system will depend on $\abeta>0$ since $\abeta$ appears in the (discrete) Brezzi constants (see \autoref{th:Brezziconstants} and \autoref{le:discinfsup}) and therefore in the constants from \autoref{eq:operatornorms_discrete}. As a result of this, the solution of \autoref{eq:sppof4} will depend on $\abeta$ as our numerical results will illustrate later on. 
\end{remark}

\section{Numerical experiments and results}\label{sec:numerics}
Our test example is Photoacoustic Tomography (PAT), as described in \autoref{pr:weak} (see \autoref{eq:patbounded}) 
on the rectangular domain $\Omega = (0,1)^d$ over the time interval $(0,T)$ with $T=\frac{1}{4}$. 
As observation domain $\Gamma$ we use the boundary of (see \autoref{fig:L})
\begin{equation} \label{eq:omega_s}
\Omega_s:= \left(\frac{1}{4},\frac{3}{4}\right)^d.
\end{equation}
The longest distance of a point $p$ in $\Omega_s$ to $\Gamma$, that is $\max\{d(p,\Gamma): p \in \Omega_s\}$, is $\frac{1}{4}$. 
Thus for information of a point to propagate to $\Gamma$ it takes $\frac{1}{4}$ time-units (because the sound speed 
equals $1$). Therefore a measurement time of $T=\frac{1}{4}$ units guarantees uniqueness \cite[Theorem 2]{SteUhl09}.

For the numerical solution we consider 2nd order B-spline spaces with equidistant knot spans and maximum continuity 
on the interval $(a,b)$, $S_{2,\ell}(a,b)$, where $\ell$ is the number of uniform refinements performed. This space 
has mesh size $h = (b-a)/2^\ell$ and smoothness $C^1(a,b)$. Moreover, the second derivative of a 2nd order spline is 
piecewise constant and thus the spline is an element in $H^2(a,b)$. Tensor product B-spline space are the tensor product of univariate B-spline spaces.

Defining 
\begin{equation*}
 Y_h := S_{2,\ell_t}(0,T)\otimes S_{2,\ell_x}(\Omega)\cap H^1_0(\Omega),
\end{equation*}
we see that because, as already stated above, every spline is two times weakly differentiable, that for all $y_h\in Y_h$
\begin{equation*}
 (\wave[y_h],\, y_h(0),\, y_h'(0))\in L^2(0,T;L^2(\Omega)) \times H^1_0(\Omega) \times L^2(\Omega).
\end{equation*}
Thus $Y_h \subset Y$ according to definition of $Y$ in \autoref{def:YY} and thus $Y_h$
is conforming. For more complex domains isogeometric analysis, see c.f. \cite{VeiBufSanVaz14,HugCotBaz05}, can be used to 
obtain smooth conformal discretization subspaces, and multi-patch domains can be dealt with with methods described in 
\cite{BirKap19} and the references within.

The space $\Lambda$ is then discretized as in \autoref{eq:LamDisc}, which gives
\begin{equation*}
 \Lambda_h = \set{ \lambda_h = (\wave[y_h], y_h'(0)) : y_h \in Y_h,\, y_h(0) = 0 }.
\end{equation*}
The resulting augmented operator $\mathcal{A}_{\alpha,\abeta,h}$ from \autoref{eq:discsaddlepointproblemoperatorform} is preconditioned with $\mathcal{P}_{\alpha,\abeta,h}$ following \autoref{th:discpcsystem},
\begin{equation*}
 \mathcal{P}_{\alpha,\abeta,h}^{-1}\mathcal{A}_{\alpha,\abeta,h}=\begin{pmatrix}P_{Y_{\alpha,\abeta,h}}^{-1} & 0\\0 & P_{\Lambda_h}^{-1}\end{pmatrix}\begin{pmatrix}A_{\alpha,\abeta,h}&B_h'\\B_h&0\end{pmatrix},
\end{equation*}
where $P_{Y_{\alpha,\abeta,h}}:Y_h\rightarrow Y_h'$ and $P_{\Lambda_h}:\Lambda_h\rightarrow \Lambda_h'$ are the matrix representations of $\inner{\cdot}{\cdot}_{Y_{\alpha,\abeta}}$ and $\inner{\cdot}{\cdot}_\Lambda$, respectively. Note that $P_{Y_{\alpha,\abeta,h}}=A_{\alpha,\abeta,h}$ by definition of the inner product $\inner{\cdot}{\cdot}_{Y_{\alpha,\abeta}}=a_{\alpha,\abeta}(\cdot,\cdot)$.

\subsection{Robustness of the preconditioner} 
In the first series of numerical experiments for Photoacoustic Tomography we study the robustness of our 
preconditioner $\mathcal{P}_{\alpha,\abeta,h}$ as introduced in \autoref{th:discpcsystem}. 

Here we use the observation surface $\Gamma=\partial\Omega_s$ (cf. \autoref{eq:omega_s}).
Condition numbers for different $\alpha$ and different numbers $\ell=\ell_t=\ell_x$ of uniform refinements are given in \autoref{t:CondEq1} and 
\autoref{t:CondEq2} for $d = 2$ and $d = 3$, respectively. 
\autoref{t:ItNumEq1} and \autoref{t:ItNumEq2} contain iteration numbers for solving the preconditioned system using the minimal 
residual method (MINRES) with zero right hand side and random initial starting vector. The stopping criteria is the reduction of 
the residual error by $10^{-8}$. As predicted from \autoref{th:discpcsystem}, the condition numbers, iteration numbers respectively, are independent of the mesh size $h$, as well as, 
the regularization parameter $\alpha$.
\begin{table}[h]
\centering 
  \begin{tabular}{| l || l | l | l | l | l |}
    \hline
    $\ell \textbackslash \alpha$   & $10^0$ & $10^{-2}$ & $10^{-5}$ & $10^{-7}$ & DoFs\\ \hline \hline
    $2$  &   2.65163 & 2.64811 & 2.64779 & 2.64770 & 176  \\  \hline   
    $3$  &   2.66313 & 2.65083 & 2.64879 & 2.64879 & 1216 \\  \hline
  \end{tabular}
  \caption{Condition numbers $\kappa\left(\mathcal{P}^{-1}_{\alpha,\abeta,h}\mathcal{A}_{\alpha,\abeta,h}\right)$ for  $d=2$ and $\abeta=1$.}
  \label{t:CondEq1}
\end{table}
\begin{table}[h]
\centering 
  \begin{tabular}{| l || l | l | l | l | l |}
    \hline
    $\ell \textbackslash \alpha$   & $10^0$ & $10^{-2}$ & $10^{-5}$ & $10^{-7}$ & DoFs\\ \hline \hline
    $2$  &   2.65112 & 2.649  & 2.64887 & 2.64587  & 704 \\  \hline
  \end{tabular}
  \caption{Condition numbers $\kappa\left(\mathcal{P}^{-1}_{\alpha,\abeta,h}\mathcal{A}_{\alpha,\abeta,h}\right)$ for $d=3$ and $\abeta=1$.}
  \label{t:CondEq2}
\end{table}

\begin{table}[h]
\centering 
  \begin{tabular}{| l || l | l | l | l | l |}
    \hline
    $\ell \textbackslash \alpha$   & $10^0$ & $10^{-2}$ & $10^{-5}$ & $10^{-7}$ & DoFs\\ \hline \hline
    $2$  &    9 &  9&  9&  9& 176  \\  \hline   
    $3$  &    9 &  9&  9&  9& 1216 \\  \hline
    $4$  &    9 &  7&  7&  7& 8960 \\  \hline
    $5$  &    7 &  7&  7&  7& 68608\\  \hline
  \end{tabular}
  \caption{Iteration numbers: $\mathcal{P}^{-1}_{\alpha,\abeta,h}\mathcal{A}_{\alpha,\abeta,h}$ for $d=2$ and $\abeta=1$.}
  \label{t:ItNumEq1}
\end{table}

\begin{table}[h]
\centering 
  \begin{tabular}{| l || l | l | l | l | l |}
    \hline
    $\ell \textbackslash \alpha$   & $10^0$ & $10^{-2}$ & $10^{-5}$ & $10^{-7}$ & DoFs\\ \hline \hline
    $2$  & 9 &  9 &  9 &  9 & 704 \\  \hline
    $3$  & 7 &  7 &  7 &  7 & 9728 \\  \hline
  \end{tabular}
  \caption{Iteration numbers: $\mathcal{P}^{-1}_{\alpha,\abeta,h}\mathcal{A}_{\alpha,\abeta,h}$ for $d=3$ and $\abeta=1$.}
  \label{t:ItNumEq2}
\end{table}

The preconditioner is not robust with respect to $\abeta$. This is reflected by the iteration numbers
in \autoref{t:ItNumEqbeta}. The number of iterations needed to reach the threshold increase significantly as $\abeta$ decreases.
\begin{table}
\centering 
  \begin{tabular}{| l || l | l | l | l |}
    \hline
    $\abeta \textbackslash \alpha$   & $10^0$ & $10^{-2}$ & $10^{-5}$ & $10^{-7}$\\ \hline \hline
    $10^{0}$  &    7  &    7 &    7 &  7   \\  \hline 
    $10^{-2}$  &  23  &   21 &   19 &  19  \\  \hline
    $10^{-5}$  &  351 &  343 &  163 &  153 \\  \hline
    $10^{-7}$  & 3039 & 2940 & 1145 &  769 \\  \hline
  \end{tabular}
  \caption{Iteration numbers: $\mathcal{P}^{-1}_{\alpha,\abeta,h}\mathcal{A}_{\alpha,\abeta,h}$ for $d=2$ and $\ell = 5$.}
  \label{t:ItNumEqbeta}
\end{table}

\subsection{Initial source recovery}
Here, we present numerical results for PAT with the preconditioning method described above using simulated data.
The ground truth, the smiley, is represent in \autoref{fig:projectp2r8}. It is constructed to be a second order 
spline with $8$ refinements, that is an element of $S_{2,8}((0,1)^2)$. 
\begin{figure}[bht]
    \centering \includegraphics[width=0.40\textwidth]{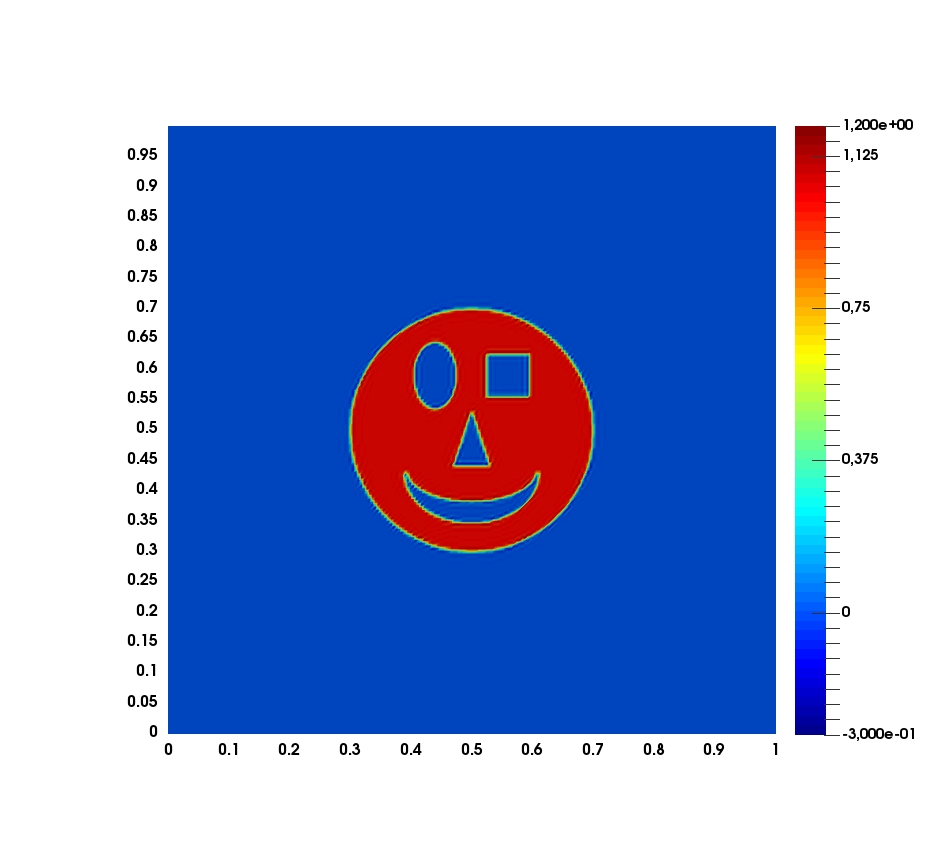}
    \caption{Ground truth $u \in S_{2,8}((0,1)^2)$.}
    \label{fig:projectp2r8}
\end{figure}
We used simulated measurement data $z_d$, which is obtained by numerically solving the wave equation with the ground truth 
with the Galerkin method in space and a finite difference method in time, with time step $h_t =  T/2^{10}$, where $T = \frac{1}{4}$.

For solving the inverse problem of PAT we discretize the augmented optimality system on the space 
\begin{equation*}
 Y_h := S_{2,6}(0,T)\otimes S_{2,6}((0,1)^2)\cap H^1_0((0,1)^2)
\end{equation*}
and $\Lambda_h$ according to \autoref{eq:LamDisc}. It is obvious that since we make only $6$ refinements, the ground truth cannot be 
recovered accurately.

The observation surface is again $\Gamma=\partial\Omega_s$ and the resulting preconditioned linear system 
of equations
\begin{equation*}
 \begin{pmatrix}P_{Y_{\alpha,\abeta,h}}^{-1} & 0\\0 & P_{\Lambda_h}^{-1}\end{pmatrix}\begin{pmatrix}A_{\alpha,\abeta,h}&B_h'\\B_h&0\end{pmatrix}\begin{pmatrix}y_h\\\lambda_h\end{pmatrix}
 =\begin{pmatrix}P_{Y_{\alpha,\abeta,h}}^{-1} & 0\\0 & P_{\Lambda_h}^{-1}\end{pmatrix}\begin{pmatrix}l_h\\0\end{pmatrix}
\end{equation*}

is solved using the MINRES method.
This is done by using sparse direct solvers for the sub-systems with the matrices $P_{Y_{\alpha,\abeta,h}}=A_{\alpha,\abeta,h}$ and $P_{\Lambda_h}$. 
This is currently the bottle neck of the numerical procedure as the direct inversion of such matrices requires a lot of memory, which limits 
us for instance to a maximum of $\ell = 6$ uniform refinements.

\autoref{fig:framea0r6b0} shows the reconstruction for $\alpha = 1.0$ and $\abeta = 1.0$. 
In \autoref{fig:framea7r6b0} we reduced $\alpha$ to $10^{-7}$ while $\abeta$ stays the same. 
Both reconstructions do not resemble the ground truth.
For obtaining the results shown in \autoref{fig:framea7r6b2}, \autoref{fig:framea7r6b5}, \autoref{fig:framea7r6b6} and \autoref{fig:framea7r6b7} 
we reduce $\abeta$ to $10^{-2}$, $10^{-5}$, $10^{-6}$ and $10^{-7}$, respectively. We see from these figures that the value $\abeta$ 
significantly effects the recovered image. Smaller values of $\abeta$ give a better recovery, however, too small values of $\abeta$ yield 
instabilities which can be observed in \autoref{fig:framea7r6b7}.
\begin{figure}[bht]
  \center
  \begin{minipage}{0.47\textwidth}
    \centering \includegraphics[width=0.85\textwidth]{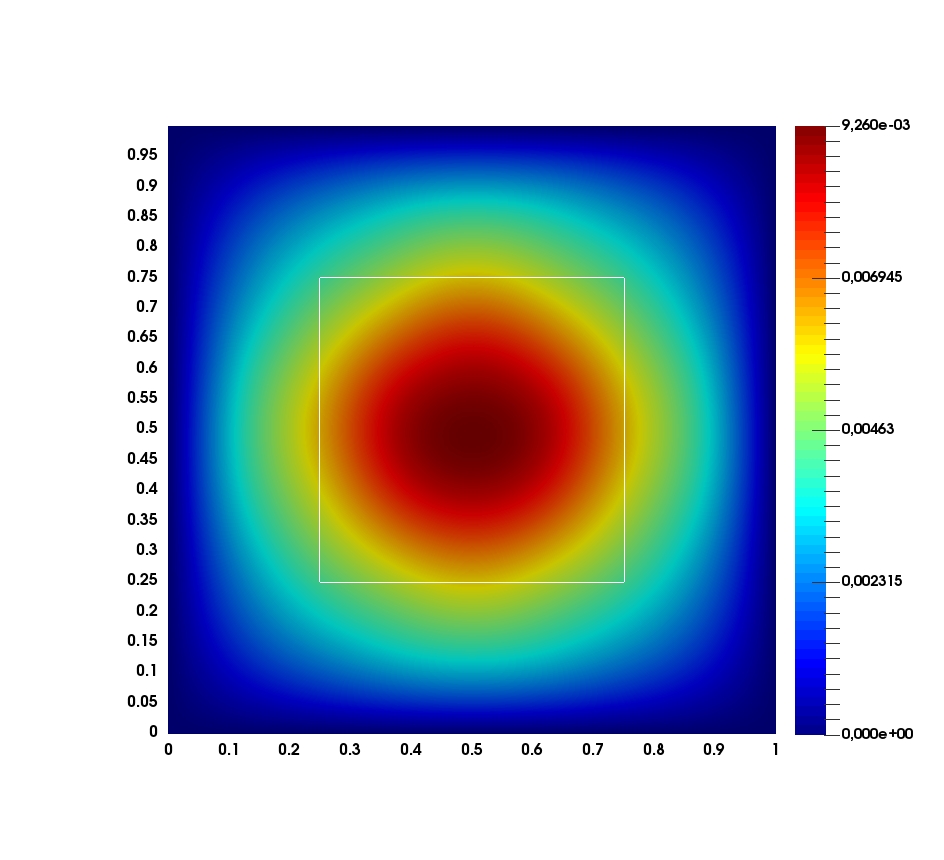}
    \caption{Recovered image: $\alpha = 1.0$ and $\abeta  = 1.0$.}
    \label{fig:framea0r6b0}
  \end{minipage}
\begin{minipage}{0.47\textwidth}
  \centering \includegraphics[width=0.85\textwidth]{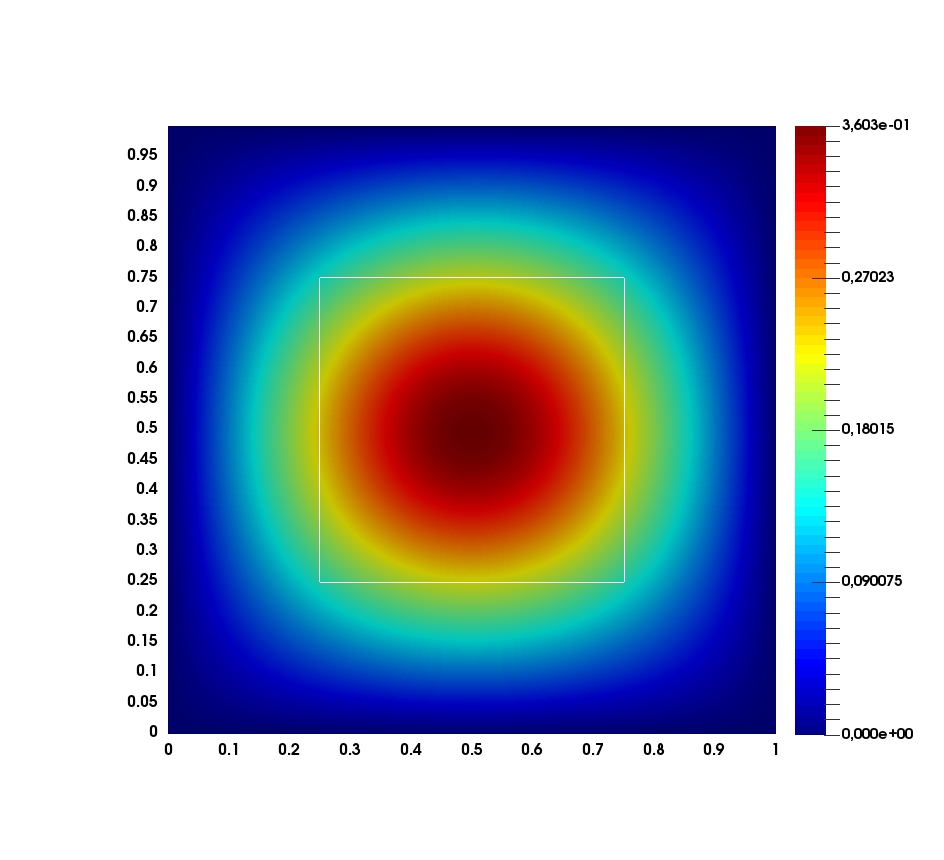}
  \caption{Recovered image: $\alpha = 10^{-7}$ and $\abeta  = 1.0$.}
  \label{fig:framea7r6b0}
\end{minipage}
\end{figure}

\begin{figure}[bht]
  \center
  \begin{minipage}{0.47\textwidth}
    \centering \includegraphics[width=0.85\textwidth]{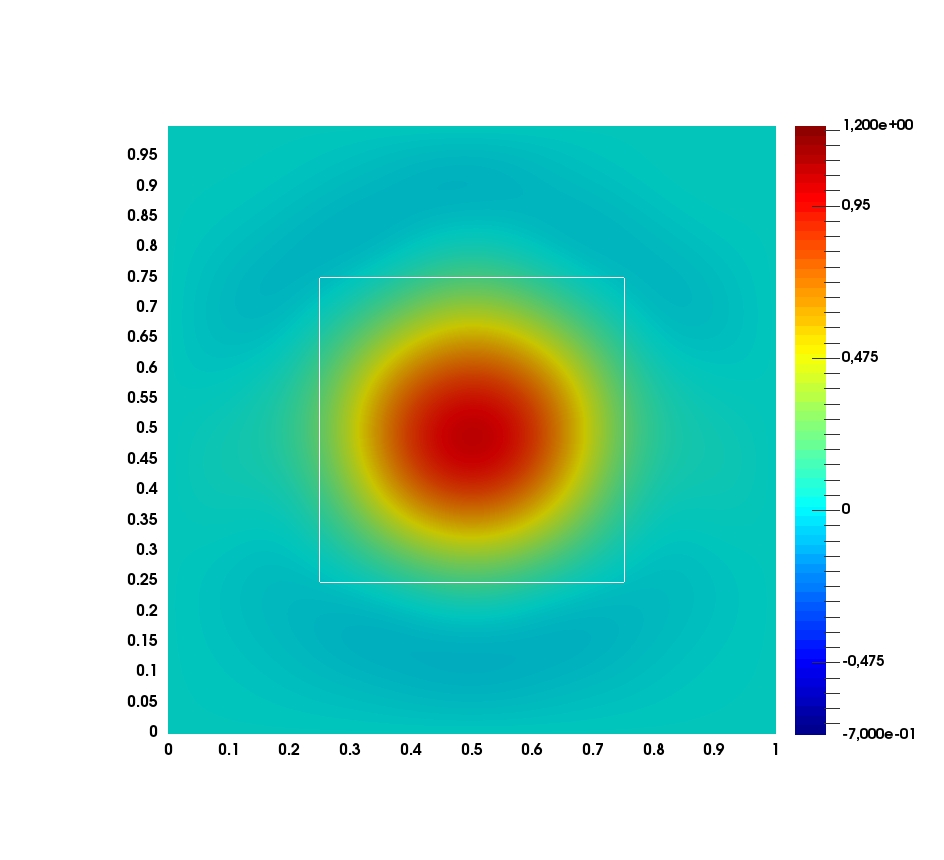}
    \caption{Recovered image: $\alpha = 10^{-7}$ and $\abeta  = 10^{-2}$.}
    \label{fig:framea7r6b2}
  \end{minipage}
\begin{minipage}{0.47\textwidth}
  \centering \includegraphics[width=0.85\textwidth]{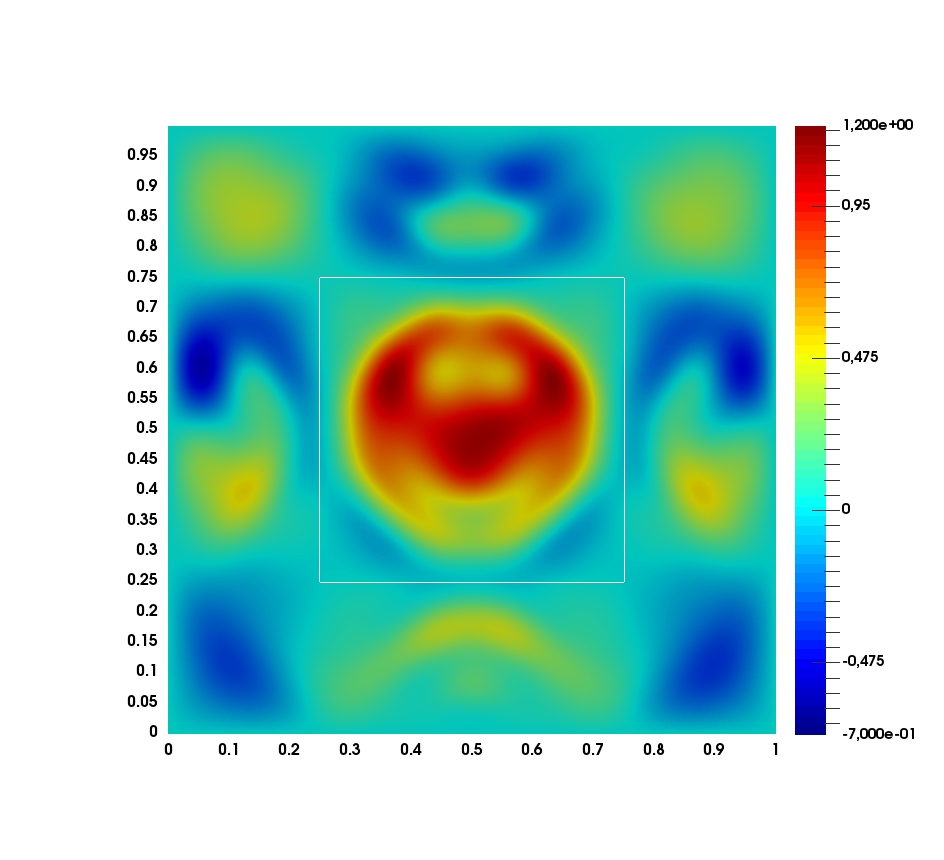}
  \caption{Recovered image: $\alpha = 10^{-7}$ and $\abeta  = 10^{-5}$.}
  \label{fig:framea7r6b5}
\end{minipage}
\end{figure}

\begin{figure}[bht]
  \center
  \begin{minipage}{0.47\textwidth}
    \centering \includegraphics[width=0.85\textwidth]{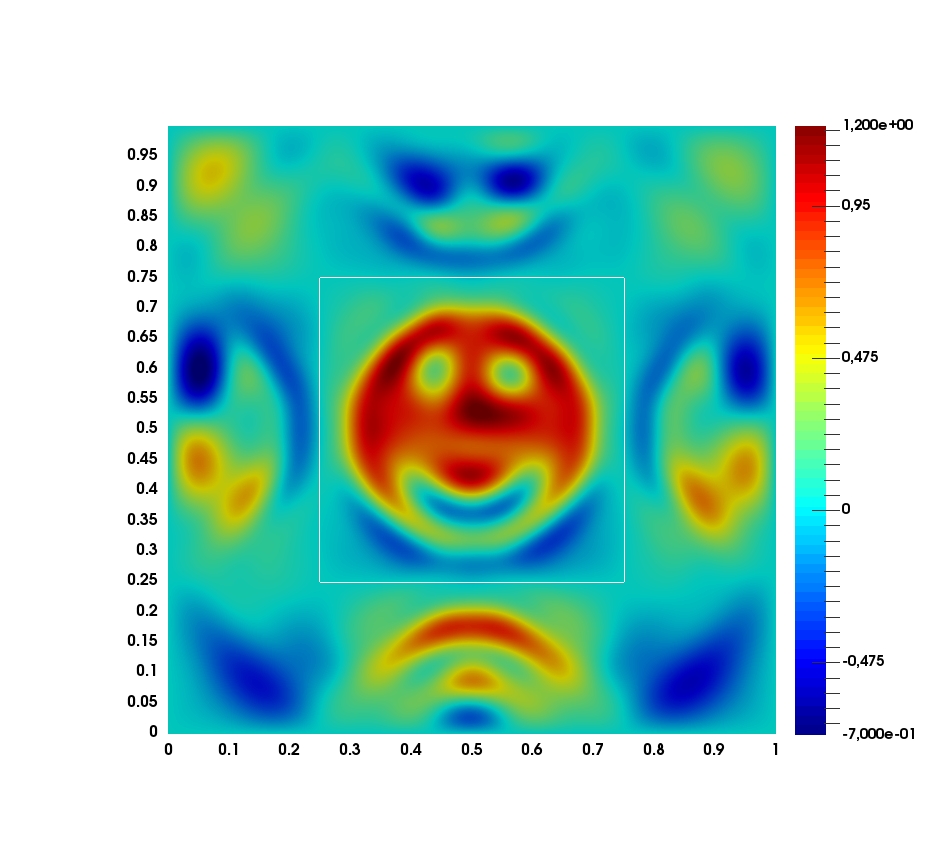}
    \caption{Recovered image: $\alpha = 10^{-7}$ and $\abeta  = 10^{-6}$.}
    \label{fig:framea7r6b6}
  \end{minipage}
\begin{minipage}{0.47\textwidth}
  \centering \includegraphics[width=0.85\textwidth]{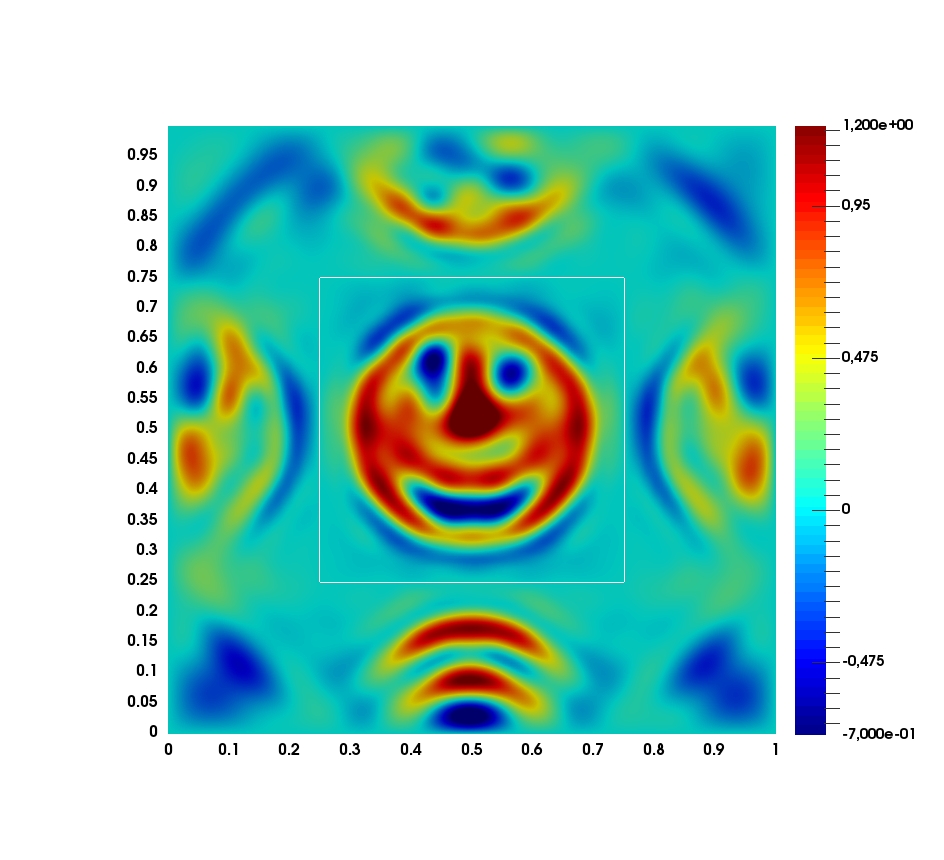}
  \caption{Recovered image: $\alpha = 10^{-7}$ and $\abeta  = 10^{-7}$.}
  \label{fig:framea7r6b7}
\end{minipage}
\end{figure}

\subsection{Discussion regarding the augmented parameter $\abeta$}
In \autoref{subsec:disc} we presented a stable discretization scheme based on an augmented Lagrangian stabilization 
and a particular choice of $\Lambda_h$.
The discrete preconditioner $\mathcal{P}_{\alpha,\abeta,h}$ from \autoref{th:discpcsystem} is not robust with respect to 
$\abeta$, see \autoref{re:rhorobustness}.
The iteration number for the preconditioned discretized augmented system $\mathcal{P}_{\alpha,\abeta,h}^{-1}\mathcal{A}_{\alpha,\abeta,h}$ 
increases as $\abeta$ goes to zero, however a small $\abeta$ is needed to recover the desired image, as shown in the figures. 
\clearpage

\section*{Acknowledgement}
OS, JS and WZ acknowledge support from the Austrian Science Fund (FWF) within the
national research network Geometry and Simulation, project S11702 and S11704. 
AB and OS are supported by I3661-N27 (Novel Error Measures and 
Source Conditions of Regularization Methods for Inverse Problems).
Moreover, OS is supported by the Austrian Science Fund (FWF), with SFB F68, 
project F6807-N36.

\section*{References}
\renewcommand{\i}{\ii}
\printbibliography[heading=none]

\end{document}